\documentclass[11pt,leqno]{amsart}
\usepackage{amsmath, amsthm, amssymb, xspace, mathrsfs}
\usepackage[breaklinks=true]{hyperref}%To get clickable links to papers
\usepackage{amsmath,amscd}
\newtheorem*{acknowledgement}{Acknowledgements}

\theoremstyle{plain}
\newtheorem{theorem}[equation]{Theorem}
\newtheorem{lemma}[equation]{Lemma}
\newtheorem{prop}[equation]{Proposition}
\newtheorem{cor}[equation]{Corollary}

\theoremstyle{definition}
\newtheorem{defin}[equation]{Definition}

\numberwithin{equation}{section}

\newcommand{\bggo}{\mathcal O}

\newcommand{\comment}[1]{}

\DeclareMathOperator{\spec}{\ensuremath{Spec}}

\DeclareMathOperator{\ad}{\ensuremath{ad}}

\begin{document}

\title[]{On the local structure of quantizations in characteristic $p$ }

\author{Akaki Tikaradze}
\address{Department of Mathematics, University of Toledo, Toledo, Ohio 43606}
\email{\tt tikar06@gmail.com}
\subjclass[2010]{Primary 16S80; Secondary 17B63  }

\begin{abstract}
Let $A$ be a central quantization of an affine Poisson variety $X$ over a field of characteristic $p>0.$
We show that the completion of $A$ with repect to a closed point $y\in X$ is isomorphic to the tensor product 
of the Weyl algebra with a local Poisson algebra. This result can be thought of as a positive characteristic
analogue of results of Losev and Kaledin about slice algebras of quantizations in characteristic 0.

\end{abstract}

\maketitle
\begin{center}
\emph{Dedicated to my mother Natalia Tikaradze}
\end{center}
\section{Introduction}
Let $\bold{k}$ be an algebraically closed field of characteristic $p>2.$
By a Poisson algebra over $\bold{k}$ we will understand a commutative $\bold{k}$-algebra equipped
with a $\bold{k}$-linear Poisson bracket $\lbrace\quad ,\quad \rbrace.$ Throughout given a Poisson algebra $B$ and $b\in B$, by
$\ad(b)$ we will denote the corresponding derivation $\lbrace b,\quad \rbrace.$ While for a non-commutative
algebra $A$ and its element $a\in A,$ by $\ad(a)$ we will denote the commutator $[a,\quad ].$
Also, for an algebra $A$ by $Z(A)$ we will denote its center.
 By a (deformation) quantization of a Poisson algebra $(B, \lbrace\quad ,\quad \rbrace)$ we will understand an associative
$\bold{k}[[\hbar]]$-algebra $A$ such that $A$ is topologically free $\bold{k}[[\hbar]]$-module, $B=A/\hbar A,$ and
$$\lbrace a\mod \hbar , b\mod \hbar \rbrace =(\frac{1}{\hbar}[a , b]) \mod \hbar,\quad a, b\in A.$$
Motivated by the notion of central quantizations introduced by Bezrukavnikov and Kaledin [\cite{BK}, Definition 1.2], we will give the following

\begin{defin}
A quantization $A$ of a Poisson $\bold{k}$-algebra $B$ is called weakly central if  
for any $a\in A$ there exists $a^{[p]}\in A$ such that
$$a^p-\hbar^{p-1}a^{[p]}\in Z(A).$$ In particular $Z(A)\mod \hbar$ contains $B^p.$
\end{defin}
The existence of a weakly central quantization of a Poisson algebra $B$ imposes on it the following
necessary condition: for any $b\in B$ there exists $b^{[p]}\in B$ such that $$\ad(b)^p=\ad(b^{[p]}).$$
This motivates the following
\begin{defin}
A Poisson $\bold{k}$-algebra $B$ is said to be a weakly restricted Poisson algebra if for
any $f\in B$ there exists $f^{[p]}\in B$ such that $$\ad(f)^p=\ad(f^{[p]}).$$
\end{defin}

Our goal is to study the local structure of a quantization algebra $A$
in terms of symplectic leaves of the affine Poisson variety $\spec B.$ Let us recall
the definition of symplectic leaves in algebraic setting.

\begin{defin} 
A symplectic leaf in an affine Poisson variety $X=\spec B$ over $\bold{k}$ is a maximal locally closed 
connected subvariety $Y\subset X$
such that the corresponding ideal of its closure $I(\overline{Y})$ is a Poisson ideal and
$Y$ is a smooth symplectic variety under the Poisson bracket induces from $B.$ 
\end{defin}

By $\overline{W_{n, \hbar}}$ we will denote the complete Weyl algebra over $\bold{k}[[\hbar]]:$
$$\overline{W_{n, \hbar}}=\bold{k}[[\hbar]]\langle x_1,\cdots, x_n, y_1,\cdots, y_n\rangle, \quad [x_i, x_j]=0=[y_i, y_j], [x_i, y_j]=\hbar\delta_{ij}$$
Similarly $W_n$ will denote the usual Weyl algebra
$$W_n=\bold{k}[x_1,\cdots, x_n,y_1,\cdots, y_n], \quad [x_i, x_j]=0=[y_i, y_j], [x_i, y_j]=\delta_{ij}$$
%Clearly $W_n$ is a subalgebra of $\overline{W_{n, \hbar}}[\hbar^{-1}].$

Let $A$ be a weakly central quantization of an affine Poisson variety $X=\spec B.$ 
 Let $y\in \spec B$ be a closed point, and $m_y$ be the
corresponding maximal ideal. Denote by $A_{\bar{y}}$ the completion of $A$ with respect to two-sided
ideal $\rho^{-1}(m_y)$-the pre-image of $m_y$  under the quotient $\rho:A\to A/\hbar A=B.$ 
Put $\rho^{-1}(m_y)\cap Z(A)=Z_y$-an ideal in $Z(A)$ containing $\hbar.$
Denote by $Z(A)_{\bar{y}}$ the completion of $Z(A)$ with respect to $Z_{y}.$
Since $Z(A)$ is Noetherian and $A$ is a finitely generated $Z(A)$-module, it follows that
$A_{\bar{y}}=A\otimes_{Z(A)}Z(A)_{\bar{y}}.$ Also, $A_{\bar{y}}$ is a quantization of $B_{\bar{y}}$-the
completion of $B$ with respect to $y.$
Let
$Y$ be a symplectic leaf in $\spec B$ containing $y.$ 
Let $\dim Y=2n.$ Our
main result below (Theorem \ref{kwak}) shows that algebra $A_{\bar{y}}$ contains the completed Weyl algebra $\overline{W_{n,\hbar}}.$
In what follows by a local Poisson $\bold{k}$-algebra we will mean a local 
$\bold{k}$-algebra equipped with a Poisson bracket such that the maximal ideal is a Poisson ideal.
More precisely,
we have the following

\begin{theorem}\label{kwak}
Algebra $A_{\bar{y}}$ is isomorphic to $\overline{W_{n,\hbar}}\otimes_{Z(\overline{W_{n,\hbar}})} A_{\hat{y}}^{+},$
where $A_{\hat{y}}^{+}$ is a weakly central quantization of a  complete local Poisson algebra.
\end{theorem}

This result is motivated and may be seen as a positive characteristic analogue of 
results of Kaledin [\cite{K}, Proposition 3.3] and Losev \cite{L} on slice
algebras of quantizations over $\mathbb{C}.$

\section{Auxiliary Lemmas}

%Thus $\overline{W_{n, \hbar}(\bold{k})}$ is a quantization of the coordinate ring of the formal affine $ 2n$-dimensional symplectic space.
%It was proved in \cite{BK} that all such quantizations must be isomorphic to the complete Weyl algebra. This result plays a key role
%in classifying quantizations of symplectic varieties in characteristic $p.$

We will recall few simple lemmas that will be used in the proof of our main result. We will include
their proofs for the reader's convenience.
\newpage
\begin{lemma}\label{matrix}
Let $S$ be a quantization of a Poisson algebra over $\bold{k}$, let $H\subset S$ be a $\bold{k}[[\hbar]]$-subalgebra
which is isomorphic to a quotient of $\overline{W_{n,\hbar}}. $
Assume that $Z(H)\subset Z(S).$ Let $S'$ be the centralizer of $H$ in $S.$ Then the multiplication map
 $H\otimes_{Z(H)} S'\to S$ is an isomorphism.
Moreover if $I\subset S$ is a two-sided ideal closed under $\frac{1}{\hbar}\ad(a),a\in S,$  then
 $I=H\otimes_{Z(H)}I',$ where $I'=S'\cap I.$ 

\end{lemma}
\begin{proof}
Proceeding by induction on $n$, it suffices to show the statement for $n=1.$
Let $H$ be generated over $\bold{k}[[\hbar]]$ by elements $x, y$ where $[x, y]=\hbar.$
By the assumption $x^p, y^p\in Z(S).$

Let us view $S$ as a module over the Weyl algebra $$W_1(\bold{k})=\bold{k}\langle \alpha, D\rangle, [\alpha, D]=1,$$ 
where the action by $\alpha$ is just the left multiplication by $x$, and $D(a)=\frac{1}{\hbar}[y, a], a\in S.$ 
Then $D^p(S)=0.$ Therefore $S$ is a left $W_1(\bold{k})/(D^p)$-module. On the other hand 
since $W_1/(D^p)$ is isomorphic to the $p\times p$ matrix algebra over $\bold{k}[\alpha^p],$
it follows that $S=k[x]\otimes_{\bold{k}[x^p]}S'$, where $S'$ is the centralizer
of $y$ in $S.$ Then $S'$ is $\bold{k}[y]$-module closed under the bracket $\frac{1}{\hbar}[ x, -].$ Arguing similarly as above
we get that $S=H\otimes_{Z(H)}S_1$, where $S_1$ is the centralizer of $H.$ Similarly $I=H\otimes_{Z(H)}I',$ where $I'=I\cap S'.$

\end{proof}

In the remainder of this section, given a ring $R,$ we will denote by $x$ and $y$ the usual
generators of the Weyl algebra over $R:$
$$W_1(R)=R\langle x, y\rangle,\quad [x, y]=1.$$

\begin{lemma}\label{identity}
Let $A$ be a weakly central quantization.
Let $f, g\in A$ be such that $[f, g]=\hbar.$ Then $$(f+g^{p-1}f^{[p]})^p-(g^{p-1}f^{[p]})^p\in Z(A).$$
\end{lemma}
\begin{proof}
Recall that in the Weyl algebra $W_1(R)$ over a commutative ring $R$ of characteristic $p$
the following identity (see for example [\cite{B}, Theorem 1.3])$$(x+ry^{p-1})^p=x^p+(ry^{p-1})^p-r, r\in R.$$
 In our case since $[\frac{1}{\hbar}f, g]=1$ and $f^{[p]}$ commutes with $f$ and $g,$
 by putting $$x=\frac{1}{\hbar}f, y=g, r=\frac{1}{\hbar}f^{[p]}$$  in the above equality we obtain that
$$(f+g^{p-1}f^{[p]})^p-(g^{p-1}f^{[p]})^p=f^p-\hbar^{p-1}f^{[p]}\in Z(A).$$

\end{proof}

\begin{lemma}\label{qow}

Let $M$ be a left  module over the Weyl algebra $W_1(\bold{k}).$
Then $$M=x^{p-1}M+yM.$$

\end{lemma}
\begin{proof}
In fact $W_1(\bold{k})=x^iW_1(\bold{k})+yW_1(\bold{k})$ for any $i\leq p-1.$
This follows immediately from induction on $i$ since $$[x^{i},\frac{1}{i} y]=x^{i-1}, 1\leq i\leq p-1.$$

\end{proof}

 \section{The proof of the main theorem}

 We have the following key result.

\begin{prop}\label{lift}
Let $(B, m)$ be a complete local  Noetherian Poisson $\bold{k}$-algebra with the maximal ideal $m$ such that
$B$ is a finitely generated module over $B^{p}.$ Let $I\subset m$ be a Poisson ideal. 
Let $A$ be a weakly central quantization of $B.$
Assume that $S=B/I$ contains elements $x_i, y_j, 1\leq i, j\leq n$ such that 
$$\lbrace x_i, y_j\rbrace=\delta_{ij}, \quad \lbrace x_i, x_j\rbrace=\lbrace y_i, y_j\rbrace=0,\quad \ad(x_i)^p=\ad(y_j)^p=0. $$
 Then there are elements $z_i, w_j\in A, 1\leq i, j\leq n$ such that they lift $x_i, y_j$ and 
 $$[z_i, w_j]=\delta_{ij}\hbar, \quad [z_i, z_j]=[w_i, w_j]=0, \quad z_i^p, w_j^p\in Z(A).$$ 
% Put $S=\bold{k}[z_i, w_j].$ Then $A=SA'$, where $A'$ is the centralizer of $S.$
\end{prop}

Here is the Poisson version of this proposition without quantizations. Its proof is essentially the same as one of Proposition \ref{lift}.

\begin{prop}\label{lift'}
Let $(B, m)$ be a complete local  weakly restricted Noetherian Poisson $\bold{k}$-algebra with the maximal ideal $m$ such that
$B$ is a finitely generated module over $B^{p}.$ Let $I\subset m$ be a Poisson ideal. 
Assume that $S=B/I$ contains elements 
$x_i, y_j, 1\leq i, j\leq n$ such that 
$$\lbrace x_i, y_j\rbrace=\delta_{ij}, \quad \lbrace x_i, x_j\rbrace=\lbrace y_i, y_j\rbrace=0,\quad \ad(x_i)^p=\ad(y_j)^p=0.$$
 Then there are elements $z_i, w_j\in B, 1\leq i, j\leq n$ such that they lift $x_i, y_j$ and 
 $$\lbrace z_i, w_j\rbrace=\delta_{ij},\quad \lbrace z_i, z_j\rbrace=\lbrace w_i, w_j\rbrace=0,\quad \ad(z_i)^p=\ad(w_j)^p=0.$$ 
% Put $S=\bold{k}[z_i, w_j].$ Then $A=SA'$, where $A'$ is the centralizer of $S.$

\end{prop}

\begin{proof}[Proof of Proposition \ref{lift}]

%\begin{lemma}
%Let $(B, m)$ be a complete noncommutative $p$-Poisson $\bold{k}$-algebra and let $I\subset B$ be a 
%Poisson ideal. Suppose that $m^pB+I$ is a two sided Poisson ideal.  
%Let $x, y\in B/I$ be such that $ \lbrace x, y\rbrace =1, x^{[p]}=0=y^{[p]}.$ Then there exists
%$f, g\in B$ lifting $x, y$ such that $\lbrace f, g\rbrace=1, f^{[p]}=0=g^{[p]}.$

%\end{lemma}
%\begin{proof}
We will proceed by induction on $n.$
Let $\rho:A\to A/\hbar A=B$ denote the quotient map.
Put $m'=\rho^{-1}(m).$ Thus $m'$ is the maximal ideal of $A$ and $A$ is complete in the $m'$-adic
topology. Put $m'^{[p]}=\lbrace x^p, x\in m'\rbrace.$
Let us put  $J'=\rho^{-1}(I)$ and $J=J'+m'^{[p]}A.$ Then $J, J'$ are two-sided ideals in $A$ and $\hbar\in J'\subset J.$
Moreover for any $a\in A$ we have $$\frac{1}{\hbar}\ad(a)(J')\subset J',\quad \frac{1}{\hbar}\ad(a)(J)\subset J.$$
Let $f_1, g_1\in A$ be arbitrary lifts of $x_1, y_1$ respectively. Therefore $$[f_1,  g_1]=\hbar+\hbar z_1,\quad  z_1\in J'.$$ 
Now given elements $f_n, g_n\in A, n\geq 1,$ such that $$x_1=f_n \mod J', y_1=g_n\mod J',\quad [f_n, g_n]=\hbar+\hbar z_n, z_n\in J^{n},$$ 
we will construct elements
$f_{n+1}, g_{n+1}\in A$ such that 
$$f_n=f_{n+1}\mod J'^{n},\quad y_1=g_{n+1}\mod J',$$
$$ [f_{n+1}, g_{n+1}]=\hbar+\hbar z_{n+1},\quad  z_{n+1}\in J^{n+1}.$$
We claim that there exist $\omega \in J'^{n+1}, z_n', z_n''\in J'^{n}$ such that $$\hbar z_n=\hbar f_n^{p-1}z_n'+[z''_n, g_n]+\hbar\omega.$$
\noindent  Indeed, notice that $J'^n/J'^{n+1}$ is a left module over the Weyl algebra $$W_1(\bold{k})=\bold{k}\langle x, y\rangle,\quad [x, y]=1,$$
where $x$ acts as the  multiplication by $f_n$, while $y$ acts as $-\frac{1}{\hbar}\ad(g_n).$ Now applying Lemma \ref{qow}
to  $J'^{n}/J'^{n+1}$ implies that $$z_n\in  f_n^{p-1}J'^n+\frac{1}{\hbar}[J'^n, g_n]+{J'}^{n+1},$$
which yields the desired result.
Let us put $f'_n=f_n-z_n''.$ Then we may write
$$[f'_n, g_n]=\hbar+\hbar (f'_n)^{p-1}w_n+\hbar\omega',\quad \omega'\in J^{n+1},\quad  w_n\in J^n.$$ 
Now recall that by our assumption on $A$, for any $a\in A$ there exists $a^{[p]}\in A$ such that
$$\ad(a^p)=\ad(a)^p=\hbar^{p-1}\ad(a^{[p]}).$$ 
\noindent Thus, we have
$$[f'_n, {\hbar}^{p-1}g_n^{[p]}]=\ad(g_n)^{p-1}(\hbar (f'_n)^{p-1}w_n+\hbar\omega')=(-1)\hbar^{p}w_n+\hbar^{p}f'_n\omega_n''+\hbar^p\eta, $$
\noindent for some $\omega_n''\in J^{n},\eta\in J^{n+1}.$ Then 
$$[f'_n, g_n+ (f'_n)^{p-1}g_n^{[p]}]=\hbar+\hbar (f'_n)^p\omega_n''+\hbar \eta',\quad (f'_n)^p\omega_n'\in J^{n+1},\quad\eta'\in J^{n+1}.$$
Since $y_1=g_n \mod J'$ and $\ad(y_1)^p=0,$ it follows that $g_n^{[p]}\in J'.$ Hence $$y_1=g_n+(f'_n)^pg_n^{[p]}\mod J'.$$
Therefore, we may put $$f_{n+1}=f'_n,\quad g_{n+1}=g_n+(f'_n)^{p-1}g_n^{[p]}.$$
We see from the construction that  the sequence $f_n, n\geq 1$ converges. Let us put $f=\lim f_n.$ Thus we have that $$[ f, g_n]=\hbar  \mod \hbar J^{n}.$$
Writing $g_n=g_1+\hbar g_n'$ and $[f, g_1]=\hbar+\hbar z_1$, we get that $$z_1=[f, -g_n'] \mod J^{n}.$$
Hence $$z_1\in \bigcap_{n}[f, A]+J^{n}.$$
It follows from our assumptions that $A$ is a finitely generated module over a complete
Noetherian ring $Z(A)$. Hence $[f, A]$ is finitely generated as a $Z(A)$-module. Then it follows from the Artin-Rees lemma
that there exists $g'\in A$ such that $[f, g']=z_1.$ Putting $g=g_1-\hbar g'$ we get that
$$[f, g]=\hbar, \quad x_1=f \mod J',\quad y_1=g\mod J'.$$

Next we want to modify elements $f, g$ such that 
$$\ad(f)^p=\ad(g)^p=0,\quad [f, g]=\hbar.$$
\noindent We will do so inductively: let $\phi_n\in A$ be such that $$\phi_n\mod J'=x_1,\quad [\phi_n, g]=\hbar,\quad \phi_n^{[p]}\in m'^n.$$ 
We have that by Lemma \ref{identity} $$(\phi_n+g^{p-1}\phi_n^{[p]})^{[p]}=(g^{p-1}\phi_n^{[p]})^{[p]}.$$
 \noindent But $(g^{p-1}\phi_n^{[p]})^{[p]}\in m'^{n+1}.$ Indeed, this
follows from the fact that if $x\in m^n,$ then $x^{[p]}\in m^n.$ 
Put $\phi_{n+1}=\phi_n+g^{p-1}\phi_n^{[p]}.$ Then since $\phi_{n}^{[p]}\in J',$ we
get that $ \phi_n=\phi_{n+1} \mod m'^nJ'$
and $\phi_{n+1}^{[p]}\in m'^{n+1}.$ Putting $\phi=\lim \phi_n$ we get that  $$[\phi, g]=\hbar,\quad \phi^{[p]}=0.$$
\noindent Proceeding in the same manner with $g$ we get elements $z_1, w_1\in A$ such that 
$$[z_1, w_1]=\hbar,\quad  x_1^{[p]}=0= y_1^{[p]},\quad z_1\mod J'=x_1, w_1\mod J'=y_1.$$

%We have that $A/p^{-1}(I)$ contains $\bold{k}[x_i, y_j]_{1\leq i\leq n,1\leq j\leq n}.$ as a Poisson subalgebra. 
%By the lemma there exists a lift $f_1, g_1\in A$ of $ x_1, y_1$ such that $$[f_1, g_1]=\hbar, f_1^{[p]}=0=g_1^{[p]}.$$
\noindent Let us put 
 $$H=\bold{k}[[\hbar]]\langle z_1, w_1\rangle,\quad \bar{z_1}=z_1\mod \hbar,\quad \bar{w_1}=w_1 \mod \hbar.$$ Thus by Lemma \ref{matrix} we have that  $A=H\otimes_{Z(H)}A_,$ where $A_1$ is the centralizer of $H$ in $A.$
It follows that $A_1$ is a weakly central quantization of the Poisson centralizer of $\bar{z_1}, \bar{w_1}$ in $B,$ which we denote
by $B_1.$ Let $J'_1=J'\cap A_1.$ Then by Lemma \ref{matrix}
 we have that $J'=H\otimes_{Z(H)}J'_1.$ In particular, $J'_1\mod \hbar=I'$ is a Poisson ideal in $B_1.$
It follows that $B_1/I'$ contains $x_i, y_j, i,j\geq 2.$ Proceeding by induction we are done.

%Moreover it follows that since $p^{-1}(I)$ is a $W_1$-submodule of $A$, we have
% $$p^{-1}(I)=\bold{k}[z_1, w_1]I_1, I_1=p^{-1}(I)\cap A_1.$$ Moreover $ A_1/I_1$ is the Poisson centralizer of $x_1, y_1$ in $ B/I.$
%Indeed $S=A_1/I_1$ is clearly in the centralizer of $x_1, y_1$ in $B/I$, moreover $B/I=\bold{k}[x_1, y_1] B_1$ where $B_1$
%is the Poisson centralizer of $x_1, y_1$ in $B/I,$ and $(B/I)^p\subset S.$
%Suppose that $t\in B/I$ belongs to the centralizer of $x_1, y_1.$ Let $t'\in A$ be its lift. 
%We know that $t'$ can be written as $t'=\sum_{i,j< p}t'_{ij}z_1^iw_1^j, t_{ij}\in A_1.$ 
%Thus we may write $t=\sum_{i, j<p}t_{ij}x^iy^j, t_{ij}\in S.$
%Then it follows immediately that $t=t_{00}\in S.$ Hence we conclude that  $A_1$ is a central quantization of $B', I'$ is a Poisson
%ideal of $B'$ and $x_i, y_j\in B'/I', i, j>1.$ Thus we may proceed by induction to complete the proof.

\end{proof}

Next we will recall the following characteristic $p$ analogue of the Darboux's theorem 
from \cite{BK}.
\begin{lemma}\label{sheep}
Let $Y$ be an $2n$-dimensional affine symplectic variety over $\bold{k}$ 
such that $\bggo_{}(Y)$ is a weakly restricted Poisson algebra.
Let $y\in Y$ be a closed point, and let $\bggo_{}(Y)_{\bar{y}}$ denote the completion
of the coordinate ring of $Y$ with respect to $y.$ Then $\bggo_{}(Y)_{\bar{y}}$  is
isomorphic to the Poisson algebra $\bold{k}[[x_1,\cdots, x_n,y_1,\cdots, y_n]]$ with
the bracket  
$$\lbrace x_i, x_j\rbrace=0=\lbrace y_i,y_j\rbrace,\quad \lbrace x_i,y_j\rbrace=\delta_{ij}.$$

\end{lemma}
\begin{proof}
Let $m$ denote the maximal ideal of $\bggo_{}(Y)_{\bar{y}}.$ Denote by $m^{[p]}$ the ideal generated
by $b^p, b\in m.$
Then it follows from results of Bezrukavnikov-Kaledin [\cite{BK}, Theorem 1.1, Proposition 3.4] that
$\bggo_{}(Y)_{\bar{y}}/m^{[p]}$ as a Poisson algebra is isomorphic to 
$\bold{k}[x_1,\cdots, x_n, y_1,\cdots, y_n]/(x_1^p,\cdots x_n^p, y_1^p,\cdots, y_n^p) $
with the Poisson bracket $$\lbrace x_i, x_j\rbrace=0=\lbrace y_i,y_j\rbrace,\quad \lbrace x_i, y_j\rbrace=\delta_{ij}.$$
 Now Proposition \ref{lift'} yields the desired result.

\end{proof}

\begin{proof}[Proof of Theorem \ref{kwak}.]

Let $I$ be the Poisson ideal in $B_{\bar{y}}$ such that $B_{\bar{y}}/I=\bggo_{}(Y)_{\bar{y}}$- the completion
of $\bggo_{}(Y)$ at $y.$ It follows that $\bggo_{}(Y)_{\bar{y}}$ is a weakly restricted Poisson algebra.
Therefore by Lemma \ref{kwak}, it follows that 
$$\bggo_{}(Y)_{\bar{y}}=\bold{k}[[x_i, y_j]], \quad1\leq i, j\leq n,$$ with the usual Poisson bracket
$$\lbrace x_i, x_j\rbrace=0=\lbrace y_i,y_j\rbrace,\quad\lbrace x_i, y_j\rbrace=\delta_{ij}.$$
Now Proposition \ref{lift}  and Lemma \ref{matrix} imply that
 $$A_{\bar{y}}=\overline{W_{n,\hbar}}\otimes_{Z(\overline{W_{n,\hbar}})}A_{\bar{y}}^{+},$$ where $A_{\bar{y}}^{+}$ is the centralizer of
$\overline{W_{n, \hbar}}$ in $A_{\bar{y}},$ hence $\hbar\in A_{\bar{y}}^{+}.$ Put $L=A_{\bar{y}}^{+}/\hbar A_{\bar{y}}^{+},$ hence $L\subset B_{\bar{y}}.$ 
Then $A_{\bar{y}}^{+}$ is a quantization of $L$ and 
$$B_{\bar{y}}=\widehat{\bggo_{}(Y)_{\bar{y}}}\otimes_{\widehat{\bggo_{}(Y)_{\bar{y}}}^p}L,$$ where 
$\widehat{\bggo_{}(Y)_{\bar{y}}}$ % \subset B_{\bar{y}},\widehat{\bggo_{}(Y)_{\bar{y}}}=\overline{W_{n,\hbar}}/\hbar 
is a lift of $\bggo_{}(Y)_{\bar{y}}$ under the quotient map
$B_{\bar{y}}\to  \bggo_{}(Y)_{\bar{y}}$. Therefore 
$L$ is the Poisson centralizer of $\widehat{\bggo_{}(Y)_{\bar{y}}}$ 
in $B_{\bar{y}}.$
Put $m'=\overline{m_{y}}\cap L.$ Then $m'$ is the maximal ideal of $L.$ We claim that $m'$ is a Poisson ideal.
Indeed the image of $m'$ in $B_{\bar{y}}/I=\bggo_{}(Y)_{\bar{y}}$ is in the center, therefore 
$$\lbrace m', m'\rbrace\subset I\cap L\subset m'.$$
Thus $A_{\bar{y}}^{+}$ is a quantization of a complete local Poisson algebra $L.$ 
Let $a\in A_{\bar{y}}^{+},$ then $a^p-\hbar^{p-1}a^{[p]}\in Z(A_{\bar{y}})\subset Z(A_{\bar{y}}^{+}).$
It follows that $a^{[p]}$ is in the centralizer of $\overline{W_{n, \hbar}},$ hence $a^{[p]}\in A_{\bar{y}}^{+}.$
Therefore $A_{\bar{y}}^{+}$ is a weakly central quantization as desired.
\end{proof}

As an immediate corollary of Theorem \ref{kwak}, one can reprove the following Kac-Weisfeiler type statement for quantizations as
in \cite{T}. 

\begin{cor}
Let $M$ be an $A$-module which is finite and free over $\bold{k}[[\hbar]],$ such that $M/\hbar M$ is supported
on $y.$ Then $\dim_{\bold{k}[[\hbar]]} M$ is a multiple of $p^{n}.$

\end{cor}
\begin{proof}

It follows that $M[\hbar^{-1}]$ is a nonzero  module over $A_{\bar{y}}[\hbar^{-1}].$ Therefore, by Theorem 
\ref{kwak} $M[\hbar^{-1}]$ is a module over $\overline{W_n}[\hbar^{-1}].$ Hence $\dim_{\bold{k}[[\hbar]]} M$ is
divisible by $p^n.$

\end{proof}

\begin{acknowledgement} 
I am very grateful to the referee for many useful suggestions that led to the improvement of the paper.

\end{acknowledgement}

\end{document}